\documentclass[reqno,11pt]{amsart}
\usepackage{setspace}
\makeatletter

\usepackage[T1]{fontenc}
\usepackage{lmodern}

\usepackage[a4paper,top=1in,bottom=1in,left=1in,right=1in,marginparwidth=1.75cm]{geometry}
\usepackage{microtype}
\usepackage{textcmds}  % amsrefs needs this, but it has to be loaded early to avoid re-defs.
\usepackage{amsmath, amssymb, amsfonts, amstext, verbatim, amsthm, mathrsfs, stmaryrd}
\usepackage{thmtools}
\usepackage{appendix}
\usepackage{mathdots}
\usepackage{ytableau}
\usepackage[all,cmtip,2cell]{xy}
\usepackage{pgf,tikz,pgfplots}
\pgfplotsset{compat=1.15}
\usetikzlibrary{arrows}
\usetikzlibrary{positioning}
\usetikzlibrary{quotes}
\usepackage{tikz-cd}
\usepackage{quiver}
\usepackage{mathtools}
\usepackage[dvipsnames]{xcolor}

\usepackage{graphicx}
\graphicspath{ {images/} }

\tikzset{
  symbol/.style={
    draw=none,
    every to/.append style={
      edge node={node [sloped, allow upside down, auto=false]{$#1$}}}
  }
}

\usepackage{enumerate}
\usepackage{enumitem}
\setlist{topsep=1em, itemsep=1em}
\usepackage[colorlinks=true,linkcolor=blue,citecolor=blue,urlcolor=blue,citebordercolor={0 0 1},urlbordercolor={0 0 1},linkbordercolor={0 0 1}]{hyperref} %needs to be loaded after most things
\usepackage[shortalphabetic]{amsrefs}
\usepackage[nameinlink, capitalize]{cleveref}
\usepackage{enotez}
\setenotez{backref=true}

\def\mydefc#1{\expandafter\def\csname c#1\endcsname{\mathcal{#1}}}
\def\mydefallc#1{\ifx#1\mydefallc\else\mydefc#1\expandafter\mydefallc\fi}
\mydefallc ABCDEFGHIJKLMNOPQRSTUVWXYZ\mydefallc

\def\mydefb#1{\expandafter\def\csname b#1\endcsname{\mathbf{#1}}}
\def\mydefallb#1{\ifx#1\mydefallb\else\mydefb#1\expandafter\mydefallb\fi}
\mydefallb ABCDEFGHIJKLMNOPQRSTUVWXYZ\mydefallb

\def\mydeffrac#1{\expandafter\def\csname frac#1\endcsname{\mathscr{#1}}}
\def\mydeffracall#1{\ifx#1\mydeffracall\else\mydeffrac#1\expandafter\mydeffracall\fi}
\mydeffracall ABCDEFGHIJKLMNOPQRSTUVWXYZ\mydeffracall

\theoremstyle{plain}
\newtheorem{thm}{Theorem}[section]
\newtheorem{cor}[thm]{Corollary}
\newtheorem{lem}[thm]{Lemma}

\newtheorem{conj}[thm]{Conjecture}
\newtheorem{prop}[thm]{Proposition}

\theoremstyle{definition}
\newtheorem{rem}[thm]{Remark}
\newtheorem{defn}[thm]{Definition}

\newtheorem{ex}[thm]{Example}

\newtheorem{introthm}{Theorem}

\crefname{thm}{theorem}{theorems}
\crefname{cor}{corollary}{corollaries}
\crefname{lem}{lemma}{lemmas}
\crefname{prop}{proposition}{propositions}
\crefname{defn}{definition}{definitions}
\crefname{rem}{remark}{remarks}
\crefname{claim}{claim}{claims}
\crefname{conj}{conjecture}{conjectures}
\crefname{ex}{example}{examples}
\crefname{fact}{fact}{facts}
\crefname{quest}{question}{questions}

\Crefname{thm}{Theorem}{Theorems}
\Crefname{cor}{Corollary}{Corollaries}
\Crefname{lem}{Lemma}{Lemmas}
\Crefname{prop}{Proposition}{Propositions}
\Crefname{defn}{Definition}{Definitions}
\Crefname{rem}{Remark}{Remarks}
\Crefname{claim}{Claim}{Claims}
\Crefname{conj}{Conjecture}{Conjectures}
\Crefname{ex}{Example}{Examples}
\Crefname{fact}{Fact}{Facts}
\Crefname{quest}{Question}{Questions}

\def\rm{\mathrm}

\def\K{\mathrm{K}}

\DeclareMathOperator{\rk}{rank}

\newcommand{\DCoh}{\mathcal{D}^{b}}

\DeclareMathOperator{\Coh}{Coh}

\DeclareMathOperator{\Cone}{Cone}

\DeclareMathOperator{\End}{End}

\DeclareMathOperator{\GL}{GL}

\DeclareMathOperator{\Hom}{Hom}

\DeclareMathOperator{\NS}{NS}

\newcommand{\norm}[1]{\left\lVert#1\right\rVert}

\DeclareMathOperator{\Stab}{Stab}

\DeclareMathOperator{\im}{im}

\DeclareMathOperator{\td}{td}

\def\bf{\mathbf}
\def\cal{\mathcal}

\newcommand{\cl}{{\cal L}}
\newcommand{\ZZ}{\mathbf{Z}}

\renewcommand{\H}{{\rm H}}
\newcommand{\Gr}{{\rm {Gr}}}

\newcommand{\QQ}{{\mathbf{Q}}}

\newcommand{\ko}{\mathcal{O}}

\newcommand{\PP}{\mathbf{P}}

\newcommand{\ch}{\mathrm{ch}}

\usepackage{pbox}
\usepackage[normalem]{ulem}

\makeatletter

\usepackage{babel}
\begin{document}

\title[Stability conditions supported on Lefschetz classes]{Stability conditions supported on Lefschetz classes}

\author[F. Giovenzana]{Franco Giovenzana}
\address{\parbox{0.9\textwidth}{Universit\'e Paris-Saclay, CNRS, Laboratoire de Math\'ematiques d'Orsay\\[1pt]
Rue Michel Magat, B\^at. 307, 91405 Orsay, France
\vspace{1mm}}}
\email{\url{franco.giovenzana@universite-paris-saclay.fr}}

\author[A. Robotis]{Antonios-Alexandros Robotis}
\address{\parbox{0.9\textwidth}{Columbia University, Department of Mathematics\\[1pt]
2990 Broadway, New York, NY
\vspace{1mm}}}
\email{\url{a.robotis@columbia.edu}}

\author[F. Rota]{Franco Rota}
\address{\parbox{0.9\textwidth}{Universit\'e Paris-Saclay, CNRS, Laboratoire de Math\'ematiques d'Orsay\\[1pt]
Rue Michel Magat, B\^at. 307, 91405 Orsay, France
\vspace{1mm}}}
\email{\url{franco.rota@universite-paris-saclay.fr}}

\author[V. Zuliani]{Vanja Zuliani}
\address{\parbox{0.9\textwidth}{Universit\'e Paris-Saclay, CNRS, Laboratoire de Math\'ematiques d'Orsay\\[1pt]
Rue Michel Magat, B\^at. 307, 91405 Orsay, France
\vspace{1mm}}}
\email{\url{vanja.zuliani@universite-paris-saclay.fr}}

\begin{abstract}
    Recently, C.~Li constructed stability conditions on the derived categories of all smooth complex projective varieties. These stability conditions satisfy the support property of Kont\-sevich--Soibelman with respect to the lattice in cohomology generated by powers of an ample class. We extend Li's results to construct stability conditions with full support on Lefschetz type varieties whose algebraic cohomology is generated by divisor classes.
    This includes all smooth complex projective varieties of dimension at most 3 and smooth projective toric varieties.
\end{abstract}

\maketitle

\tableofcontents

\addtocontents{toc}{\protect\setcounter{tocdepth}{1}}

\linespread{1.15}\selectfont

\section{Introduction}

Stability conditions on triangulated categories were introduced by Bridgeland \cite{Br07} over twenty years ago as a mathematical interpretation of Douglas' theory of $\Pi$-stability in physics \cite{DouglasDbranes}. Since then, stability conditions have found a plethora of applications in algebraic geometry via wall-crossing techniques --- see \cites{BMunreasonable,ProjbiratBM,Mukaisprogram} among many other examples. All the while, however, a fundamental gap remained in the theory: there was no general existence result for stability conditions on the derived categories of smooth projective varieties. Though progress was made over the years on this foundational question, there were no systematic existence results available outside of dimensions $1$ and $2$ \cites{ABLStabilitysurf,Br07,BridgelandK3,Macricurves}, and some examples in dimension $3$ \cites{stabilityonFanothreefolds,LiQuintic}.

Recent groundbreaking work of Chunyi Li \cite{ChunyiLiRemark} proved the existence of stability conditions on $\DCoh(X)$ for $X$ any smooth complex projective variety, by constructing stability conditions on $\bP^n$ which can be pulled back to $X$ along a closed immersion $X\to \bP^n$. A desirable feature of Li's construction is that the resulting stability conditions are \emph{geometric} in that all skyscraper sheaves of points on $X$ are stable of the same phase. 

The modern definition of a stability condition \cite{Bayer_short} requires the \emph{support property}, as formulated by Kontsevich--Soibelman \cite{KS08} --- see Definition~\ref{D:stability}. To state the support property for $\DCoh(X)$, one chooses a surjective homomorphism $v:\rm{K}_0(X)\to \Lambda$, where $\Lambda$ is a finite rank free Abelian group and one defines a corresponding space $\Stab_{(\Lambda,v)}(X)$ of stability conditions. Bridgeland's deformation theorem \cite{Br07}, rephrased as in \cite{Bayer_short}, implies that $\Stab_{(\Lambda,v)}(X)$ has a canonical structure of a complex manifold, locally modeled on $\Hom(\Lambda,\bC)$. In particular, the higher the rank of the lattice $\Lambda$, the greater the number of deformation directions in $\Stab_{(\Lambda,v)}(X)$.

It is reasonable to wonder if there is a canonical choice of homomorphism $v:\K_0(X) \to \Lambda$ for a given variety $X$. Physical reasoning as in \cite{AspinwallDouglasDbrane}*{p. 9} suggests that the central charge $Z:\K_0(X) \to \bC$ of stability conditions on $X$ should depend only on the Chern classes of sheaves on $X$. In particular, a natural and maximal choice of lattice is given by $\cH^\bullet(X):=\im(\ch: \K_0(X) \to  \H^\bullet_{\rm{alg}}(X)_{\bQ})$. Thus, we say that a stability condition has \emph{full support} if it satisfies the support property with respect to $\ch:\K_0(X) \to \cH^\bullet(X)$.  

By contrast, the lattice $\Lambda_H$ used in \cite{ChunyiLiRemark} to formulate the support property can be identified with subgroup of $\H^\bullet(X,\bQ)$ generated by the powers of the hyperplane class $H$ pulled back along the closed immersion $X\to \bP^n$. The rank of $\Lambda_H$ is typically small compared to that of $\cH^\bullet(X)$. 

Let $L^\bullet_X$ denote the integral subalgebra of $\H^\bullet(X,\bQ)$ generated by the divisor classes, called the Lefschetz algebra, and let $v_X := \ch(-)\cup \sqrt{\td(X)}$. We write $\cH^\bullet_M(X) := \sqrt{\td(X)} \cH^\bullet(X) = \im(v_X)$. The main results of the present work are summarized as:

\begin{introthm}
[= Theorem~\ref{T:ellipticconditions} + Corollary~\ref{C:maincorollary}]
    Let $X$ be a smooth complex projective variety. There exist a closed immersion $i:X\to \prod_{i=1}^k \bP^{n_i} =: \bP^\chi$ and geometric stability conditions $\sigma$ on $\DCoh(X)$ which satisfy the support property with respect to $\im(i_* \circ v_X:\cH^\bullet_M(X)\to \cH^\bullet_M(\bP^\chi))$. 
\label{T:introthmA}
\end{introthm}

We say that a variety $X$ is of \emph{Lefschetz type} if the cup product restricted to the Lefschetz algebra $L_X^\bullet \otimes_{\bZ} \bQ $ is non-degenerate, in other words, $L_X^\bullet \otimes_{\bZ} \bQ$ satisfies Poincaré duality. 
Theorem~\ref{T:introthmA} has the following consequence: 

\begin{introthm}[= Corollary~\ref{C:fullsupport}]
\label{T:introthm_B}
    If $X$ is a smooth complex projective variety of Lefschetz type for which $L_X^\bullet \otimes_{\bZ} \bQ =\H^\bullet_{\rm{alg}}(X)_{\bQ}$, then $\DCoh(X)$ admits geometric stability conditions with full support.
\end{introthm}

In Example~\ref{ex:examples}, we collect classes of varieties for which Theorem~\ref{T:introthm_B} applies. Among these are all smooth projective varieties of dimension $\le 3$ and all smooth projective toric varieties. 

\subsection*{Relation to NMMP}
One of the goals of the Noncommutative Minimal Model Program (NMMP) of Halpern-Leistner \cite{NMMP} is to obtain ``canonical'' semi\-orthogonal decompositions of $\DCoh(X)$ for $X$ a smooth projective variety, by studying paths in its space of stability conditions. 
Motivated partially by physical heuristics as in  \cite{AspinwallDouglasDbrane}, Halpern-Leistner's formulation requires the existence of stability conditions on $\DCoh(X)$ with full support --- see \cite{NMMP}*{Conj. I}. The present work verifies this conjecture for many important examples.

In the context of the NMMP, a natural reason for considering stability conditions defined with respect to $\ch:\K_0(X) \to \cH^\bullet(X)\subset \H^\bullet_{\rm{alg}}(X)_{\bQ}$ is that the inclusion $\H^\bullet_{\rm{alg}}(X)_{\bC} \hookrightarrow \H^\bullet(X,\bC)$ allows one to define paths in the space of central charges $\Hom(\H^\bullet_{\rm{alg}}(X),\bC)$ using the quantum differential equation of $X$ --- see \cite{GGI}. These paths are conjectured to lift to quasi-convergent paths in the sense of \cite{quasiconvergence} which then give rise to decompositions of $\DCoh(X)$. See the works \cites{karube,KRZ,zuliani} for instances of this principle.

\subsection*{Related work}
The day before posting the first version of the present work we learned of work of Y. Cheng, which appeared on the day of posting, which shows that the mass-Hom bound, introduced in \cite{augmented}, implies the support property for numerical stability conditions \cite{YCheng}*{Thm. 2.1}. This, combined with the recent work \cite{2026stab_cond_modulispaces}, implies the support property for numerical stability conditions on all projective schemes over a field. 

For smooth and projective varieties for which the standard conjectures are known to hold, the Chern character descends to an isomorphism $\K_{\rm{num}}(X)_{\bQ} \to \H^\bullet_{\rm{alg}}(X)_{\bQ}$. Since we use the veracity of Standard conjecture A to verify the ``Lefschetz type'' condition in Example~\ref{ex:examples}, it follows that Theorem~\ref{T:introthm_B} can essentially be deduced from \cite{YCheng}*{Thm. 2.1}. Nevertheless, we feel that the present note is worth publishing since the techniques used are different than Cheng's and provide an elementary proof of the full support property, albeit in a restricted class of examples.

\subsection*{AI statement} Though we used Large Language Models for literature searches, all of the contents of this article were conceived of and written by the authors. % --- all em dashes are our own.

\subsection*{Acknowledgements} The authors thank Yiran Cheng for advanced notice of his article. We also thank Emanuele Macr\`{i} for many enlightening conversations about stability conditions and for several useful corrections to the first version of this work.

A.R. was supported by NSF grant DMS-250340. F. R. was supported by the European Union’s Horizon 2020 Research and Innovation Programme under the Marie Skłodowska-Curie grant agreement n. 101147384 (\href{https://cordis.europa.eu/project/id/101147384}{CHaNGe}).
F.G. and V.Z. were supported by ERC Synergy Grant 854361 HyperK. 

A.R. thanks Maria Teresa for her support and understanding during the writing of this paper. 

\section{Preliminaries}
Let $\cD$ be a $\bC$-linear triangulated category. To fix notation and terminology, we recall the notion of a stability condition on $\cD$.

\begin{defn}
\label{D:stability}
    A \emph{slicing} $\cP$ on $\cD$ is a collection of full additive subcategories $\{\cP(\phi)\}_{\phi \in \bf{R}}$ of $\cD$ such that: \vspace{-2mm}
    \begin{enumerate}
        \item $\phi_1>\phi_2\Rightarrow \Hom_{\cD}(\cP(\phi_1),\cP(\phi_2)) = 0$, \vspace{-2mm}
        \item $\cP(\phi)[1] = \cP(\phi+1)$ for all $\phi \in \bf{R}$, and \vspace{-2mm}
        \item for every non-zero object $E$ of $\cD$ there exists a sequence of real numbers $\phi_1>\cdots>\phi_n$ and a sequence of morphisms $0 = E_0 \to E_1\to \cdots \to E_n = E$ such that $\Cone(E_{i-1}\to E_i) \in \cP(\phi_i)$ for each $i=1,\ldots, n$. This sequence of morphisms is referred to as a \emph{Harder--Narasimhan filtration} of $E$. \vspace{-2mm}
    \end{enumerate}
    
    A \emph{pre-stability condition} on $\cD$ is a pair $(Z,\cP)$ where $\cP$ is a slicing and $Z\in \Hom_{\bf{Z}}(\K_0(\cD),\bf{C})$ is called the \emph{central charge} such that for all $\phi \in \bf{R}$ and all non-zero $E\in \cP(\phi)$ we have 
    \[
        Z(E) \in \bf{R}_{>0}\cdot \exp(i\pi \phi).
    \]
    Such an object $E$ is called \emph{semistable} of phase $\phi$.

    Let $\Lambda$ be a free Abelian group of finite rank, i.e. a lattice and let $v:\K_0(\cD)\to \Lambda$ be a rationally surjective group homomorphism.\footnote{Rationally surjective means that the induced map $\K_0(\cD)_\bQ\to \Lambda_\bQ$ is surjective.} A \emph{stability condition supported on} $(\Lambda,v)$ is a pre-stability condition $\sigma=(Z,\cP)$ such that $Z$ factors through $v$ and such that there is a norm $\norm{\:\cdot\:}$ on $\Lambda_{\bR}$ such that
    \begin{equation}
        \inf_{0\ne E\in \cP(\theta)} \frac{\lvert Z(E)\rvert}{\lVert v(E)\rVert} > 0.
    \end{equation}
\end{defn}

We denote by $\Stab_{(\Lambda,v)}(\cD)$ the set of stability conditions supported on $v:\rm{K}_0(\cD) \to \Lambda$. By Bridgeland's deformation theorem \cites{Bayer_short,Br07}, $\Stab_{(\Lambda,v)}(\cD)$ has the structure of complex manifold such that the map $\Stab_{(\Lambda,v)}(\cD) \to \Hom(\Lambda,\bC)$ given by $(Z,\cP)\mapsto Z$ is a local biholomorphism. In particular, the dimension of $\Stab_{(\Lambda,v)}(\cD)$ as a complex manifold equals the rank of $\Lambda$.

The lattice $\cH^\bullet(X):=\im(\ch:\rm{K}_0(X) \to \H^\bullet_{\rm{alg}}(X)_{\bQ})$ gives an integral structure on $\H^\bullet_{\rm{alg}}(X)_{\bQ}$, i.e. it is an Abelian subgroup such that $\cH^\bullet(X)\otimes \bQ = \H^\bullet_{\rm{alg}}(X)_{\bQ}.$ 

\begin{lem}
\label{L:changeoflattice}
    Let $\Lambda$ be an integral structure on $\H^\bullet_{\rm{alg}}(X)_{\bQ}$. There exists $\alpha\in \GL(\H^\bullet_{\rm{alg}}(X)_{\bQ})$ such that $\alpha\cdot \cH^\bullet(X) = \Lambda$. Furthermore, $\sigma = (Z,\cP)$ is supported on $(\cH^\bullet(X),\ch)$ if and only if it is supported on $(\Lambda,\alpha\circ \ch).$
\end{lem}

\begin{proof}
    Existence of $\alpha$ follows from $\cH^\bullet(X)\otimes \bQ = \H^\bullet_{\rm{alg}}(X)_{\bQ} = \Lambda\otimes \bQ$. First, $Z$ factors through $\ch:\rm{K}_0(X) \to \cH^\bullet(X)$ if and only if it is in the image of $\ch^*:\Hom_{\bZ}(\cH^\bullet(X),\bC) \to \Hom_{\bZ}(\rm{K}_0(X),\bC)$. Thus, $Z$ also factors through $\alpha\circ \ch$ since it is in the image of $(\alpha\circ \ch)^*$; this is immediate from the fact that $\alpha^*: \Hom_{\bZ}(\Lambda,\bC) \to \Hom_{\bZ}(\cH^\bullet(X),\bC)$ is an isomorphism. It is now an exercise to verify that the support property condition for $\cH^\bullet(X)$ transfers to $\Lambda$.
\end{proof}

\begin{defn}
\label{D:fullsupport}
    Let $X$ be a smooth and projective variety over $\bC$. We say that a pre-stability condition $\sigma = (Z,\cP)$ on $\DCoh(X)$ is a \emph{stability condition with full support} if it supported on some (equivalently any) pair $(\Lambda,\alpha \circ \ch)$ as in Lemma~\ref{L:changeoflattice}.
\end{defn}

\begin{ex}
\label{E:twistchern}
    Typical examples appearing in the literature are of the form $\alpha(-) = \gamma\cup(-)$ for $\gamma \in \H^\bullet_{\rm{alg}}(X)_\bQ$. For example, taking $\gamma = \sqrt{\td(X)}$ gives $\Lambda = \sqrt{\td(X)}\cH^\bullet(X)$ the \emph{Mukai lattice}. We denote this by $\cH^\bullet_M(X)$ and write $v_X(-) = \ch(-)\cup \sqrt{\td(X)}$ for the \emph{Mukai vector}. The Mukai lattice is functorial under pushforwards: if $f:X\to Y$ is a morphism of smooth projective varieties, then by the Grothendieck--Riemann--Roch theorem we have a commutative diagram:
        \[
            \begin{tikzcd}
                \rm{K}_0(X) \arrow[r,"f_*"] \arrow[d,"v_X",swap]& \rm{K}_0(Y)\arrow[d,"v_Y"]\\
                \H^\bullet(X,\bQ) \arrow[r,"f_*"]& \H^\bullet(Y,\bQ).
            \end{tikzcd}
        \]
        Thus, the Gysin map $f_*$ restricts to a homomorphism $\cH^\bullet_M(X)\to \cH^\bullet_M(Y)$. This property will be used in Section~\ref{S:constructing}.
\end{ex}

\begin{rem}
\label{R:changeoflattice}

%\FR{here could be a place to say we drop the $(-)$ from the notation in $\Stab$ (and similarly also the subscript in the Mukai vector)}
When convenient, we will drop $X$ from the notation and write $\Stab_{(\cH_M^\bullet,v)}(X)$ and $\Stab_{(\cH^\bullet,\ch)}(X)$ to denote stability conditions supported on $(\cH_M^\bullet(X),v_X)$ and $(\cH^\bullet(X),\ch)$. 

By Lemma~\ref{L:changeoflattice} and Lemma~\ref{E:twistchern}, there is an \emph{equality} of sets $\Stab_{(\cH_M^\bullet,v)}(X) = \Stab_{(\cH^\bullet,\ch)}(X)$. However, we may still write $\Stab_{(\cH^\bullet_M,v)}(X)$ to emphasize the choice of Mukai homomorphism.
\end{rem}

\begin{defn}{\cite{Lirealreduction}}
    Given $\sigma,\omega\in \Stab_{(\Lambda, v)}(\mathcal{D})$, we write \vspace{-2mm}
    \begin{enumerate}
        \item $\sigma \prec \omega$ if $\cP_{\sigma}(\theta)\subset \cP_{\omega}(<\theta)$ for all $\theta \in \bR$, and \vspace{-2mm}
        \item $\sigma \preceq \omega$ if $\cP_{\sigma}(\theta)\subset \cP_{\omega}(\leq \theta)$ for all $\theta\in \bR$. 
    \end{enumerate}
\end{defn}

A key tool used in \cite{ChunyiLiRemark} to construct stability conditions on projective varieties is the procedure of inducing a stability condition along an exact functor, as introduced in \cite{MMSinducing}. We consider only the (derived) functors $f^*:\DCoh(X) \to \DCoh(Y)$ and $f_*:\DCoh(Y) \to \DCoh(X)$ where $f:Y\to X$ is a morphism of smooth projective varieties.

\begin{comment}
    
Consider $\sigma = (Z,\cP) \in \Stab_{(\cH^\bullet,\ch)}(Y)$. We define $f_{\#}(\sigma) = (f_{\#}Z,f_{\#}\cP)$, where $f_{\#}Z = Z\circ f^* \in \Hom(\rm{K}_0(X),\bf{C})$ and $f_{\#}\cP$ is a collection of full additive sub\-categories $f_{\#}\cP(\theta)\subset \DCoh(X)$, where 
\[
    f_{\#}\cP(\theta) := \left\{E\in \DCoh(X):f^*E \in \cP(\theta) \setminus \{0\}\right\} \cup \{0\}
\]
for each $\theta \in\bf{R}$. Note that for any non-zero $E\in f_{\#}\cP(\theta)$ we have $f_{\#}Z(E) = Z(f^*E) \in \bf{R}_{>0}\cdot \exp(i\pi \theta).$

\begin{prop}
\label{P:pushforward}
    If $f_{\#}\sigma$ is a pre-stability condition and $f^*:\cH^\bullet(X) \to \cH^\bullet(Y)$ is injective, then $f_{\#}\sigma$ has full support. 
\end{prop}

\begin{proof}
    First, note that naturality of the Chern character, i.e. $f^*\circ \ch_X = \ch_Y\circ f^*$, implies that $f_{\#}Z$ factors through $\ch_X$. For the support property, equip $\cH^\bullet(Y)_{\bR}$ with a norm $\lVert\:\cdot\:\rVert_Y$ and let $\lVert \:\cdot \:\rVert_X$ denote the norm on $\cH^\bullet(X)_{\bR}$ induced along $f^*$. Then,
    \[
        \inf_{0\ne E\in f_{\#}\cP(\theta)} \frac{\lvert f_{\#}Z(E)\rvert}{\lVert \ch_X(E)\rVert} = \inf_{0\ne E \in f_{\#}\cP(\theta)} \frac{\lvert Z(f^*(E))\rvert}{\lVert \ch_Y(f^*E)\rVert} > 0
    \]
    where we used the fact that $E\in f_{\#}\cP(\theta)$ if $f^*(E) \in \cP(\theta) \setminus \{0\}$ and injectivity of $f^*$ so that $\ch_Y(f^*E) = f^*\ch_X(E)$ is non-zero. Recall also that $\ch_X(E)\ne 0$ by stability of $E$.
\end{proof}

\end{comment}

Let $\Lambda$ be a lattice equipped with a surjective map
\[
    v\colon \K_0(Y) \xrightarrow{\ch_Y} \cH^\bullet(Y) \xrightarrow{\alpha} \Lambda.
\] 
Consider $\sigma = (Z,\cP) \in \Stab_{(\Lambda,v)}(Y)$. We define $f_{\#}(\sigma) = (f_{\#}Z,f_{\#}\cP)$, where $f_{\#}Z = Z\circ f^* \in \Hom(\rm{K}_0(X),\bf{C})$ and $f_{\#}\cP$ is a collection of full additive sub\-categories $f_{\#}\cP(\theta)\subset \DCoh(X)$, where 
\[
    f_{\#}\cP(\theta) := \left\{E\in \DCoh(X):f^*E \in \cP(\theta) \setminus \{0\}\right\} \cup \{0\}
\]
for each $\theta \in\bf{R}$. Note that for any non-zero $E\in f_{\#}\cP(\theta)$ we have $f_{\#}Z(E) = Z(f^*E) \in \bf{R}_{>0}\cdot \exp(i\pi \theta).$

\begin{prop}
\label{P:pushforward}
    If $f_{\#}\sigma$ is a pre-stability condition and $\alpha\circ f^*:\cH^\bullet(X) \to \Lambda$ is injective, then $f_{\#}\sigma$ has full support. 
\end{prop}

\begin{proof}
    First, note that naturality of the Chern character, i.e. $f^*\circ \ch_X = \ch_Y\circ f^*$, implies that $f_{\#}Z$ factors through $\ch_X$. For the support property, equip $\Lambda_{\bR}$ with a norm $\lVert\:\cdot\:\rVert_Y$ and let $\lVert \:\cdot \:\rVert_X$ denote the norm on $\cH^\bullet(X)_{\bR}$ induced along $\alpha\circ f^*$. Then,
    \[
        \inf_{0\ne E\in f_{\#}\cP(\theta)} \frac{\lvert f_{\#}Z(E)\rvert}{\lVert \ch_X(E)\rVert_{X}} 
        =
        \inf_{0\ne E \in f_{\#}\cP(\theta)} \frac{\lvert Z(f^*(E))\rvert}{\lVert v(f^*E)\rVert_Y} > 0
    \]
    where we used the fact that $E\in f_{\#}\cP(\theta)$ if $f^*(E) \in \cP(\theta) \setminus \{0\}$ and injectivity of $\alpha \circ f^*$ so that $v(f^*E) = (\alpha\circ f^*)\ch_X(E)$ is non-zero. Recall also that $\ch_X(E)\ne 0$ by stability of $E$.
\end{proof}

Next, we consider the dual operation. Given $\tau = (W,\cQ) \in \Stab_{(\cH^\bullet_M,\ch)}(X)$, we define $f^{\#}\tau = (f^{\#}Z,f^{\#}\cQ)$ where $f^{\#}Z = Z\circ f_* \in \Hom(\rm{K}_0(Y),\bC)$ and $f^{\#}\cQ$ is a collection of full additive sub\-categories $f^{\#}\cQ(\theta)\subset \DCoh(Y)$, where 
\[
    f^{\#}\cQ(\theta) := \left\{E\in \DCoh(Y):f_*E \in \cQ(\theta)\setminus \{0\}\right\} \cup \{0\},
\] 
for each $\theta \in\bf{R}$. Note again that for any non-zero $E\in f^{\#}\cQ(\theta)$ we have $f^{\#}W(E) = W(f_*E) \in \bf{R}_{>0}\cdot \exp(i\pi \theta).$

\begin{prop}
\label{P:pullback}
    Given $\tau \in \Stab_{(\cH^\bullet_M,v)}(X)$, if $f^{\#}\tau$ is a pre-stability condition and $f_*:\cH^\bullet_M(Y) \to \cH^\bullet_M(X)$ is injective, then $f^{\#}\tau$ has full support.
\end{prop}

\begin{proof}
    The proof is analogous to that of Proposition~\ref{P:pushforward}, except that in this case we use the naturality of the Mukai vector under pushforwards: $f_*\circ v_X = v_Y\circ f_* $.
\end{proof}

\section{Stability conditions on \texorpdfstring{$\bf{P}^\chi$}{Pchi}}

Consider a partition $\chi$ of $n$, i.e. a tuple $(n_1,\ldots,n_k)$ of positive integers such that $n = n_1+ \cdots+n_k$. We let $\bf{P}^\chi := \bf{P}^{n_1}\times \cdots \times \bf{P}^{n_k}$ and $\bf{P}^{1,n}:= (\bf{P}^1)^n$. Following the techniques of \cite{ChunyiLiRemark}, we construct stability conditions on $\bf{P}^\chi$. Associated to $\chi$ is the reducible root system of type $B_\chi = B_{n_1} \sqcup \cdots \sqcup B_{n_k}$. For any $n\in \bN$, $\mathsf{S}_n$ denotes the corresponding symmetric group. We write $\mathsf{S}_\chi = \mathsf{S}_{n_1}\times \cdots \times \mathsf{S}_{n_k}$ and denote the associated Weyl group by 
\[
    \mathsf{W}_{B_\chi} := (\bf{Z}/2\bf{Z})^n \rtimes \mathsf{S}_\chi. 
\]
The main observation of this section is that the results of \cite{ChunyiLiRemark} can be generalized from $B_n$ to $B_\chi$. First of all, fix an elliptic curve $E$ and consider the diagram 
\[
    \begin{tikzcd}
        E^n \arrow[r,"h"]\arrow[dr,"\pi",swap]& \bf{P}^{1,n}\arrow[d,"f"]\\
        &\bf{P}^\chi
    \end{tikzcd}
\]
where the maps are as follows. The elliptic curve $E$ has an involution $x\mapsto -x$ which defines a $(\bf{Z}/2\bf{Z})$-action. The map $h$ is the quotient map of $E^n$ by $(\bf{Z}/2\bf{Z})^n$ acting by the product of the involution action. The morphism $f$ is the quotient of $\bf{P}^{1,n}$ by $\mathsf{S}_\chi$ and $\pi := f\circ h$. 

Following \cite{Lietal}, for each $i=1,\dots, n$  fix a closed point $q_i\in E$ and define a divisor
\begin{equation}
    H_i:=E\times \cdots\times \{q_i\}\times \cdots \times E, \text{ for each } i=1,\dots, n\;\;\text{ and }\;\; H:=\sum_i H_i.\footnote{We use the same notation for $H_i,H$, the associated line bundles, and their first Chern classes.}
\end{equation}
In \emph{loc. cit.}, the authors construct geometric stability conditions on $E^n$ supported on $v: \rm{K}_0(E^n)\twoheadrightarrow \bZ^{2^n}\eqqcolon \Lambda$ defined by
\begin{equation*}
    [F]\mapsto (H_{j_1}\cdots H_{j_k}\ch_{n-k}(F))_{1\leq j_1<\cdots< j_k\leq n}.
\end{equation*}
Observe that $v$ factors through $\ch_{E^n}$: we write $v = \alpha \circ \ch_{E^n}$. The authors then prove the following result.

\begin{thm}[\cite{ChunyiLiRemark}*{Prop. 5.4} \& \cite{Lietal}*{Thm 4.5}]
\label{thm_stab_on_En}
    Fix $N\in \bZ_{\geq 1}$.
    %and a line bundle $\cL$ on $E^n$ which is numerically equivalent to $H$.
    For any $(a,b)\in \bR_{>0}\times\bR$ with $a,\lvert b\rvert \gg 1$ there exists a geometric stability condition 
    \begin{equation}\label{eq_stabN_ell}
    \begin{split}
        \sigma^{a,b} & = (\cP^{a,b},Z^{a,b})\in \Stab_{(\Lambda,v)}(E^n) \text{ with }\\
        Z^{a,b}(-) & = -\int_{E^n}e^{-(b+\sqrt{-1}a)H}\ch(-) 
    \end{split}
    \end{equation}
    %\FR{I changed $\cH^\bullet$ to $\cH^\bullet(E^n)$ in the above}
    such that \vspace{-2mm}
    \begin{enumerate} 
        \item $\phi_{\sigma^{a,b}}(\ko_p)=1\text{ for all }p\in E^n$ and \vspace{-2mm}
        \item $\sigma^{a,b}\otimes \ko(kH) \prec \sigma^{a,b}[1] \;\;\text{ for all } \;1\leq k\leq N.$\vspace{-2mm}
    \end{enumerate}  
\end{thm}

A key observation in \cite{ChunyiLiRemark} is that stability conditions like those of Theorem~\ref{thm_stab_on_En} can be used to induce stability conditions on $\bP^{1,n}$ with some remarkable properties.

\begin{prop}
    For any $(\bZ/2\bZ)^n$-invariant $\sigma\in \Stab_{(\Lambda,v)}(E^n)$, there is an induced stability condition $h_{\#}\sigma \in \Stab_{(\cH^\bullet,\ch)}(\bP^{1,n})$.
\label{P:pffromelliptic}
\end{prop}

\begin{proof}
    See \cite{ChunyiLiRemark}*{Lem. 6.1} for the proof that $h_{\#} \sigma$ is a pre-stability condition.
    Our goal is then to apply Proposition~\ref{P:pushforward} to $h_{\#} \sigma$: it suffices to check that $\alpha\circ h^*:\cH^\bullet(\bP^{1,n}) \to \Lambda$ is injective. First, note that for the standard quotient map $E\to \bP^1$, the induced map $\cH^\bullet(\bP^1) \to \cH^\bullet(E)$ is injective since it sends $\ch(\cO_p) \mapsto 2\ch(\cO_p)$ and $\ch(\cO_{\bP^1}) \mapsto \ch(\cO_E)$. 

    Next, consider the K\"unneth decompositions $\H^\bullet(\bP^{1,n},\bQ) = \bigotimes_{i=1}^n \H^\bullet(\bP^1,\bQ)$ and $\H^\bullet(E^n,\bQ) = \bigotimes_{i=1}^n \H^\bullet(E,\bQ)$. Since $h$ is an $n$-fold product of $E\to \bP^1$, we conclude that $h^*$ is injective, being an $n$-fold tensor product of injective maps. Finally, observe that the image of $h^*$ lies in the subspace $V$ generated by $\{H_{j_1}\cdots H_{j_k}\}_{1\le j_1<\cdots<j_k\le n}$. Since $\alpha$ is precisely the duality pairing on $V$, it follows that $\alpha \circ h^*$ is injective and the result now follows from Proposition~\ref{P:pushforward}.
\end{proof}

    %Next, write K\"unneth decompositions for  $\bP^{1,n}$ and $E^n$: since $h$ is an $n$-fold product of the map $E\to \bP^1$, we conclude that $h^*$ is injective. Finally, observe that the image of $h^*$ lies in the span of $\{H_{j_1}\cdots H_{j_k}\}_{1\le j_1<\cdots<j_k\le n}$: the map $\alpha$ is precisely the duality pairing on this span, hence $\alpha\circ h^*$ is injective. The result now follows from Proposition~\ref{P:pushforward}.

\begin{cor}\label{cor_from_En_to_P1n}
    Given $\sigma^{a,b}\in \Stab_{(\Lambda,v)}(E^n)$ as in \eqref{eq_stabN_ell}, then $\sigma^{a,b}$ satisfies the assumption of Proposition~\ref{P:pffromelliptic}. The induced stability condition
    \begin{equation}
        h_{\#}\sigma^{a,b}\in \Stab_{(\cH^\bullet, \ch)}(\PP^{1,n})
    \end{equation}
    \begin{enumerate}
        \item is $\mathsf{S}_\chi$-invariant; and \vspace{-2mm}
        \item satisfies $h_{\#}\sigma^{a,b}\preceq (h_{\#}\sigma^{a,b}) \otimes \cL$ for any effective line bundle $\cL$ on $\bP^{1,n}$.\vspace{-2mm}
    \end{enumerate}
\end{cor}

\begin{proof}
    To apply Proposition~\ref{P:pffromelliptic}, we must check that $\sigma^{a,b}$ is $(\bZ/2\bZ)^n$-invariant. As noted in the proof of \cite{ChunyiLiRemark}*{Thm. 6.3},
    the action of $\mathsf{W}_{B_\chi}$ fixes $Z^{a,b}$. By \cite{FLZ}*{Thm. 1.1}, all stability conditions on $E^n$ are geometric and by \cite{Lietal}*{Thm. 1.1} the map $\Stab_{(\Lambda,v)}(E^n) \to \Hom(\Lambda,\bC) \times \bR$ given by $(Z,\cP)\mapsto (Z,\phi_\sigma(\cO_p))$ is injective, where $p\in E$ is any closed point. 
    The group $\mathsf{W}_{B_\chi}$ preserves central charges and since it acts through automorphisms, it preserves $\phi_{\sigma^{a,b}}(\cO_p)$.  
    %Thus, the $\mathsf{W}_{B_\chi}$-action preserves $\phi_{\sigma^{a,b}}(\cO_p)$ 
    Thus,  $\sigma:=\sigma^{a,b}$ is $\mathsf{W}_{B_\chi}$-invariant and descends to $\bP^{1,n}.$

    Next, $\sigma \preceq \sigma \otimes h^*\cL_i$ by \cite{ChunyiLiRemark}*{Lem. 2.3}, where $\cL_i := p_i^*\cO_{\bP^1}(1)$ and $p_i:\PP^{1,n}\to \PP^1$ is the $i^{\rm{th}}$ projection. Therefore, 
    \begin{equation*}
        h_{\#}\sigma\preceq h_{\#}(\sigma\otimes h^*\cL_i)=(h_{\#} \sigma)\otimes \cL_i
    \end{equation*}
    where the ``$\preceq$'' is by \cite{ChunyiLiRemark}*{Lem 3.3} and the ``$=$'' is by \cite{ChunyiLiRemark}*{Lem. 3.2}. All effective line bundles on $\bP^{1,n}$ are of the form $\cL_{i_1}^{\otimes a_1}\otimes\cdots \otimes \cL_{i_p}^{\otimes a_p}$ for $1\le i_1< \cdots< i_p\le k$ and $a_1,\ldots,a_p\geq 0$, so by \cite{ChunyiLiRemark}*{Lem. 2.2(4)} we have $h_{\#}\sigma\preceq (h_{\#}\sigma)\otimes \cL$ for any effective line bundle $\cL$ on $\bP^{1,n}$.
\end{proof}

\begin{prop}\cite{ChunyiLiRemark}*{Prop. 6.2}\label{prop6.2}
    Let $\sigma \in \Stab_{(\cH^\bullet,\ch)}(\bf{P}^{1,n})$ be given such that \vspace{-2mm}
    \begin{enumerate}
        \item $\sigma$ is $\mathsf{S}_\chi$-invariant; and \vspace{-2mm}
        \item $\sigma \preceq \sigma \otimes \cL$ for every effective line bundle $\cL$ on $\bf{P}^{1,n}$.\vspace{-2mm}
    \end{enumerate}
    Then, $f_{\#}\sigma$ is a stability condition on $\bf{P}^\chi$ with full support. 
\end{prop}

\begin{proof}
    We consider $Y_i:=\bP^{1,n_i}$ for all $i=1, \dots, k$, the quotient maps $f_i:Y_i\to \PP^{n_i}$, and their products $Y=\prod_i Y_i$ and $f:Y\to X:=\prod_i\PP^{n_i}$. Moreover, we let 
    \begin{equation*}
        Z:=Y\times_X Y=\prod_i Y_i\times_{X_i} Y_i \subset \prod_i Y_i\times Y_i = Y\times Y.
    \end{equation*}
We will use the notation $p_2^i: Y_i\times Y_i\to Y_i$ for the projection to the second factor, $p_2 = p_2^1\times \cdots \times p_2^k$, and $\Gamma_{w_i}:=
\{(w_iy,y): y\in Y_i\}\subset Y_i\times Y_i$ for $w_i\in \mathsf{S}_{n_i}$.
    
    By \cite{ChunyiLiRemark}*{Cor 4.2}, the structure sheaf $\ko_Z\in \Coh(Y\times Y)$ admits a filtration with subquotients
    \begin{equation}
        (\ko_{\Gamma_{w_1}}\otimes p_2^{1,*} \cl^{-1}_{w_1})\boxtimes\cdots \boxtimes (\ko_{\Gamma_{w_k}}\otimes p_2^{k,*} \cl^{-1}_{w_k})=\ko_{\Gamma_{(w_1,\dots,w_k)}}\otimes p_2^*(\cl_{w_1}^{-1}\boxtimes\cdots\boxtimes \cl_{w_k}^{-1})
    \end{equation}
    where  $\cl_{w_i}$ is an effective line bundle on $Y_i$. Therefore, for any object $E\in \DCoh(Y)$ the object $p_{2*}(p_1^*E \otimes \cO_Z)$ 
    has a filtration whose subquotients are of the form 
    \begin{equation}\label{eq_filtr}
        (w_1,\dots,w_k)^*E\otimes \cl^{-1}_{(w_1,\dots,w_k)},\;\;\;w=(w_1,\dots,w_k)\in \mathsf{S}_\chi.
    \end{equation}
    For $E\in \cP_{\sigma}(\theta)$, the $\mathsf{S}_\chi$-invariance of $\sigma$ implies that $w^*E \in \cP_\sigma(\theta)$ for any $w\in \mathsf{S}_\chi$. Since, by hypothesis, $\sigma\preceq \sigma \otimes\cl$ for any effective line bundle $\cl$ on $Y$, \cite{ChunyiLiRemark}*{Lem. 2.2} gives 
\begin{equation}
    \phi^+_{\sigma}(w^*E\otimes \cl^{-1}_w)=
    \phi^+_{\sigma\otimes \cl_w}(w^*E)\leq \phi_{\sigma}(w^*E)=\theta.
\end{equation}
From filtration \eqref{eq_filtr}
it follows that $f^*f_*E\in \cP_{\sigma}(\leq \theta)$ and thus by \cite{ChunyiLiRemark}*{Prop. 3.4}, the pair $f_{\#}\sigma$ is a stability condition on $X$.

For the full support property, note that $f^*:\H^\bullet(\bP^\chi,\bQ)\to \H^\bullet(\bP^{1,n},\bQ)$ is injective, since it induces an isomorphism $\H^\bullet(\bP^\chi,\bQ)\to \H^\bullet(\bP^{1,n},\bQ)^{\mathsf{S}_\chi}$. Thus, we conclude by Proposition~\ref{P:pushforward}.
\end{proof}

Henceforth, we write $\cO_{\bP^\chi}(1) = \ko_{\PP^\chi}(1,\dots,1).$ The following is a slight generalization of \cite{ChunyiLiRemark}*{Thm. 6.3}. The main modification in our treatment is the emphasis of the lattice used to define the support property.

\begin{thm}
\label{T:ellipticconditions}
    Fix $N\in \bZ_{\geq 1}$. Then, there exists $C_N \in \bR_{>0}$ such that for any $(a,b)\in \bR_{>0}\times\bR$ with $a,\lvert b\rvert \geq C_N$ the stability condition $\sigma^{a,b}\in \Stab_{(\cH^\bullet,\ch)}(E^n)$ descends to $\bP^\chi$. That is, $\sigma:=\pi_{\#}\sigma^{a,b}\in \Stab_{(\cH^\bullet, \ch)}(\bP^\chi)$. Furthermore, $\sigma$ is geometric and satisfies: \vspace{-2mm}
 \begin{enumerate}
     \item $\sigma\preceq \sigma\otimes\ko_{\PP^\chi}(1)$ (Bayer property), \vspace{-2mm}
     \item $\sigma\otimes \ko_{\PP^\chi}(N)\prec \sigma[1]$ (restriction-N property). 
 \end{enumerate}
\end{thm}

\begin{proof}
    We follow the proof of of \cite{ChunyiLiRemark}*{Cor. 6.4}; see also \cite{2026stab_cond_modulispaces}*{Lem. 4.8}. By definition, $\pi_{\#}\sigma^{a,b}=f_{\#}(h_{\#}\sigma^{a,b})$. Hence, by Corollary~\ref{cor_from_En_to_P1n} and Proposition~\ref{prop6.2} we have $\sigma:=\pi_{\#}\sigma^{a,b}\in\Stab_{(\cH^\bullet,\ch)}(\bP^\chi)$. We next check the Bayer and restriction-N properties. For $a,\lvert b\rvert\gg 1$ we have 
    \begin{equation*}
        \sigma^{a,b}\otimes \ko(kH) \prec \sigma^{a,b}[1] \text{ for all }1\leq k\leq 2N,
    \end{equation*}
    by Theorem~\ref{thm_stab_on_En}. Moreover, by  Corollary~\ref{cor_from_En_to_P1n}(2)
    \begin{equation*}
        \begin{split}
            \sigma=\pi_{\#}\sigma^{a,b}\preceq  \pi_{\#}(\sigma^{a,b}\otimes \ko(2H))=\pi_{\#}(\sigma^{a,b}\otimes \pi^*\ko_{\PP^\chi}(1))=\sigma\otimes \ko_{\PP^\chi}(1)\\
            \sigma\otimes \ko_{\PP^\chi}(N)=\pi_{\#}(\sigma^{a,b}\otimes \ko(2NH))\prec \pi_{\#}(\sigma^{a,b}[1])=\sigma[1].
        \end{split}
    \end{equation*}
    Finally, all skyscraper sheaves on $E^n$ lie in $\cP_{\sigma^{a,b}}(1)$.
For any skyscraper $\ko_p$ of $p\in \bP^\chi$, we have $\pi^*\ko_p\in \cP_{\sigma^{a,b}}(1)$. Thus, by definition of $\pi_{\#}$, we have $\ko_p\in \cP_{\sigma}(1)$.
\end{proof}

\section{Constructing stability conditions with full support}
\label{S:constructing}

In this section, we give a mild generalization of \cite{ChunyiLiRemark}*{Thm. 6.5} and apply it to $\bf{P}^\chi$. Recall that given a morphism $f:X\to Y$ of smooth projective varieties, there is an induced homomorphism $f_*:\cH^\bullet_M(X)\to \cH^\bullet_M(Y)$ --- see Example~\ref{E:twistchern}. We write $\cH^\bullet_M(f):= \im(f_*:\cH^\bullet_M(X)\to \cH^\bullet_M(Y))$. 

\begin{thm}[{\cite{ChunyiLiRemark}*{Thm. 6.5} \& \cite{2026stab_cond_modulispaces}*{Prop. 3.21}}]
\label{T:maintheorem}
    Consider a closed immersion $i:X\hookrightarrow Y$ of smooth projective varieties and a very ample line bundle $\ko(1)$ on $Y$. There exists $N_0$ such that the following holds.
    
    For any geometric  $\sigma\in \Stab_{(\cH^\bullet_M,v)}(Y)$ satisfying: \vspace{-2mm}
    \begin{enumerate}
        \item $\sigma\preceq \sigma\otimes\ko(1)$ (Bayer property) and\vspace{-2mm}
        \item $\sigma\otimes \ko(N) \preceq \sigma[1]$ %(restriction-N property) 
        for some $N\geq N_0$, \vspace{-2mm}
    \end{enumerate}  
 the pullback $i^{\#}\sigma$ defines a geometric stability condition on $\DCoh(X)$ supported on $(\cH_M^\bullet(i), i_*\circ v_X)$.
\end{thm}

\begin{proof}
We let $n:=\dim Y$. Let us consider a resolution of $i_*\ko_X$
\begin{equation}
    0\to F\to \ko(-a_{n-1})^{\oplus m_{n-1}}\to \cdots\to \ko(-a_1)^{\oplus m_1}\to \ko_Y\to i_*\ko_X\to 0
\end{equation}
where $F\in \Coh(Y)$ and $a_j,m_j\geq 0$. 
Choose $N\in \bN$ large enough that $jN\geq a_j$ for all $1\leq j\leq n-1$. To prove that $i^{\#}\sigma$ is a pre-stability condition, we apply \cite{Polishchuk}*{Lem. 2.2.2}; thus, we only need to verify that $i_*\ko_X\otimes\cP_{\sigma}(\theta)\subseteq \cP_{\sigma}[\theta,\infty).$

For $\theta\in (0,1]$ and $E\in \cP_{\sigma}(\theta)$, the object $i_*\ko_X\otimes E\in \DCoh(Y)$ admits a resolution with quotients 
\begin{equation*}
    E,\;\; E\otimes\ko_Y(-a_j),\;\; E\otimes F[n],
\end{equation*}
hence 
\begin{equation}
\label{E:boundmin}
    \phi^-_{\sigma}(i_*\ko_X\otimes E)\geq \min\{\theta, \phi^-_{\sigma}(E\otimes\ko_Y(-a_j)), \phi^-_{\sigma}(E\otimes F[n])\}.
\end{equation}
Thus, we need to prove that all three quantities in the minimum in \eqref{E:boundmin} above are $\ge \theta$. By the Bayer property, the restriction-N property, and \cite{ChunyiLiRemark}*{Lem. 2.2(4)}, we get
\begin{equation*}
    \sigma\otimes\ko(a_j)\preceq \sigma\otimes \ko(jN) \preceq \sigma\otimes\ko(jN-N)[1]\preceq \cdots  \preceq \sigma[j].
\end{equation*}
In particular, $\sigma\preceq \sigma\otimes\ko(-a_j)[j]=:\omega$. Observe that $E':=E\otimes\ko(-a_j)[j]$ is $\omega$-semistable, hence by \cite{ChunyiLiRemark}*{Lem. 2.2(3)} we obtain 
\begin{equation*}
    \phi_{\omega}(E') \leq \phi^-_{\sigma}(E').
\end{equation*}
By definition, $\phi_{\omega}(E')=\phi_{\sigma} (E)=\theta$, and thus $\phi^-_{\sigma}(E\otimes\ko(-a_j)[j])\geq \theta.$ 

It remains to verify that $\phi^-_{\sigma}(E\otimes F[n])\geq \theta$. Since $\sigma$ is geometric, we may assume that for all $y\in Y$ we have $\ko_y\in \cP_{\sigma}(1)$. The same reasoning as in \cite{BridgelandK3}*{Lem. 10.1} implies that 
\begin{equation*}
    E\in \cP_{\sigma}(\theta)\subseteq \Coh(Y)[0,n-1],\;\;\;\;\; \Coh(Y)\subseteq \cP_{\sigma}(1-n,1].
\end{equation*}
As $F\in \Coh(Y)$ we get that
\begin{equation*}
    E\otimes F[n]\in \Coh(Y)[n,3n-1]\subseteq \cP_{\sigma}(1,3n],
\end{equation*}
in particular $\phi^-_{\sigma}(E\otimes F[n])>1\geq \theta$.

Next, by definition, for any $x\in X$ we have that $i_*\ko_x=\ko_{i(x)}\in \cP_{\sigma}(1)$ and hence 
$\ko_x\in \cP_{i^{\#}\sigma}(1)$. Thus, $i^{\#}\sigma$ is a geometric pre-stability condition on $\DCoh(X)$. Moreover, for any non-zero $i^{\#}\sigma$-semistable object $E$ we have that $i_*E$ is semistable and 
\begin{equation*}
    \frac{|i^{\#}Z(E)|}{\norm{i_*(v_X(E))}}=\frac{|Z(i_*E)|}{\norm{v_Y(i_*E)}}
\end{equation*}
where $\lVert\:\cdot\:\rVert$ is any norm on $\cH^\bullet_M(Y)$. Thus, $i^{\#}\sigma$ is supported on $(\cH^\bullet(i), i_*\circ v_X)$.
\end{proof}

Let $X$ be a smooth projective variety, and let $(\cL_1,\ldots,\cL_k)$ be very ample line bundles on $X$. To them we associate a closed immersion $i:X \hookrightarrow \bf{P}^\chi \coloneqq  \prod_{j=1}^k \bf{P}^{\dim|\cL_j|}$.

\begin{cor}
\label{C:maincorollary}
    Let $X$ be a smooth projective variety. In the above notation, the category $\DCoh(X)$ admits geometric stability conditions supported on $(\cH^\bullet_M(i),i_*\circ v_X)$.
\end{cor}
\begin{proof}
Consider the closed immersion 
    %The strongly ample vector bundle \FR{what is a very ample vector bundle? Is this standard terminology?}\alekos{changed} $\cL_1\oplus \cdots \oplus \cL_r$ gives a closed immersion 
    $i:X \to \bP^\chi$.
    %, where $\chi = \sum_{i=1}^k \dim \lvert \cL_i\rvert$. 
Theorem~\ref{T:ellipticconditions} gives geometric stability conditions $\sigma$ on $\DCoh(\bP^\chi)$ with full support satisfying the Bayer property and restriction-N property, for any fixed $N \ge 1$. Thus, Theorem~\ref{T:maintheorem} applied to $i:X\to \bP^\chi$ yields the result.
\end{proof}

We will see that in many cases there is a natural identification $\H^\bullet_{\rm{alg}}(X)_\QQ\cong\cH^\bullet_M(i)_\QQ$ and the stability condition in Corollary~\ref{C:maincorollary} will give a geometric stability condition with full support --- see Corollary~\ref{C:fullsupport}.

\begin{lem}\label{L:restrictioninjective}
    Let $i:X\to Y$ be a closed immersion of smooth projective varieties. If the ambient cohomology $i^*\H^\bullet(Y, \bQ)$ satisfies Poincar\'e duality, then the restriction of the map $i_*:\H^\bullet(X,\bQ)\to \H^\bullet(Y,\bQ)$ to the image of $i^*$ is injective. 
\end{lem}

\begin{proof}
    Let $\alpha = i^*(\alpha')$ be given, for $\alpha' \in \H^k(Y,\bQ)$. By the projection, formula $i_*\alpha = i_*(i^*\alpha' \cup 1) = \alpha' \cup \eta_X$, where $\eta_X \in \H^{2c}(X,\bQ)$ is the Poincar\'{e} dual class to the fundamental class of $X$ in $Y$ and $c = \dim Y - \dim X$. In particular, $i_*$ is a graded map of degree  $2c$ on $i^*\H^{\bullet}(Y,\bQ)$ and we can verify injectivity degree by degree. If $\alpha = i^*(\alpha')$ is non-trivial, then by Poincar\'e duality there exists $\beta\in \H^{2\dim X - k}(X,\bQ)$ such that $\int_X \beta \cup i^* \alpha'\not = 0.$ By hypothesis, we can assume $\beta = i^*\beta'$ for some $\beta'\in \H^{2\dim X - k} (Y, \bQ)$, so that by projection formula we get
    \[
    0\neq\int_X \beta \cup i^* \alpha' = \int_Y i_*i^*(\alpha'\cup \beta') = \int_Y \eta_X \cup\alpha'\cup \beta'.
    \]
     Hence $\eta_X \cup\alpha' \not = 0$ and the claim follows.
\end{proof}

Let $X$ be a complex projective variety and consider the integral subalgebra $L_X^\bullet \subset \H^\bullet(X,\bQ)$ generated by the divisor classes. 
In particular, $L_X^\bullet\otimes_{\ZZ} \QQ  =  \bf{Q}\langle \NS(X)\rangle$ is a graded subalgebra of $\H^\bullet_{\rm{alg}}(X)_{\bQ}$, which is sometimes called the \emph{Lefschetz algebra} of $X$ --- cf. \cite{Lefschetzclasses}.

\begin{rem}
Pushforward need not be injective on ambient cohomology as we explain here. Huh and Wang costructed a smooth projective variety $X$ whose Lefschetz algebra does not satisfy Poincar\'e duality, which has the following characteristics \cite{Lefschetzclasses}*{\S 3.1}:
\[
\dim X=5; \quad \rho_X =2; \quad \dim L_X^2\otimes_\bf Z \bf Q = 3; \quad \dim  L^3_X\otimes_\bf Z \bf Q = 4.
\]
Consider the embedding $i\colon X \to \bP^\chi$, then $i^*\H^6(\bP^\chi,\mathbf Q) = L^3_X\otimes_\bf Z \bf Q$ which has dimension 4, and $i_*$ has target $\H^{2\dim X -4}(\bP^\chi, \bf Q) \simeq \H^{4}(\bP^\chi, \bf Q)$ which has dimension 3. We conclude that the restriction to the ambient cohomology of $i_*$ has at least a one-dimensional kernel.
\end{rem}

\begin{defn}
    A smooth projective variety $X$ is called of \emph{Lefschetz type} if the
    Lefschetz algebra  $L_X^\bullet \otimes_{\bf{Z}} \bf{Q}$ satisfies Poincar\'{e} duality, i.e.
    the restriction of the cup product  is non-degenerate.
\end{defn}

Note that a smooth projective variety $X$ for which Standard conjecture A (see \cite{Kleimanstandard}) holds and such that $L^\bullet_X \otimes_{\ZZ}\QQ=\H^\bullet_{\rm{alg}}(X)_{\QQ}$ is of Lefschetz type.

\begin{cor}
\label{C:fullsupport}
    Any smooth projective variety $X$ of Lefschetz type such that $L^\bullet_X \otimes_{\ZZ}\QQ=\H^\bullet_{\rm{alg}}(X)_\QQ$ admits a geometric stability condition with full support.
\end{cor}

\begin{proof}
    Consider very ample line bundles $\cL_1,\ldots, \cL_k$ such that $\QQ\langle \ch_1(\cL_j)\rangle_{j=1,\dots, k}=L_X^\bullet\otimes_{\ZZ} \QQ$ and the associated closed immersion $i:X\to \bP^\chi$. It follows from Lemma~\ref{L:restrictioninjective} that $i_*:L^\bullet_X\otimes_\ZZ \QQ \to \H^\bullet(\bP^\chi,\bQ)$ is injective: indeed by the construction of $i$ we have  $i^*\H^\bullet(\bP^\chi,\bQ) =L^\bullet_X\otimes_\ZZ \QQ$.
    Therefore, $i_* :\H^\bullet_{\rm{alg}}(X)_\QQ=L^\bullet_X\otimes_{\bf{Z}} \bQ\to \H^\bullet(\PP^\chi,\QQ)$ is injective.
    We conclude that
    $i_*:\cH_M^\bullet(X) \to \cH_M^\bullet(\bP^\chi)$ is also injective and thus the stability conditions constructed in Corollary~\ref{C:maincorollary} have full support by Proposition~\ref{P:pullback}.
\end{proof}

\begin{ex}
\label{ex:examples}
    The following are examples of varieties of Lefschetz type  satisfying  $L^\bullet_X \otimes_{\ZZ}\QQ=\H^\bullet_{\rm{alg}}(X)_\QQ$. 
    By Corollary~\ref{C:fullsupport}, they admit geometric stability conditions with full support: \vspace{-2mm}
    \begin{enumerate}
        \item All projective curves and surfaces are of Lefschetz type. Thus, Theorem~\ref{T:maintheorem} gives an alternative construction of geometric stability conditions supported on the numerical Grothen\-dieck group for algebraic surfaces, as in \cites{ABLStabilitysurf,BridgelandK3}. Indeed, this follows from the observation that the Chern character induces a homomorphism $\K_{\rm{num}}(X)\to \H^\bullet_{\rm{alg}}(X)_{\bQ}$ which is a rational isomorphism for $\dim X \le 2$. \vspace{-2mm}
        \item Let $X$ be a smooth and projective threefold. The Hard Lefschetz theorem gives an iso\-morphism $\H^{1,1}(X,\bQ) \to \H^{2,2}(X,\bQ)$ coming from cupping with the class of any ample divisor hence $L^\bullet_X \otimes_{\ZZ}\QQ=\H^\bullet_{\rm{alg}}(X)_\QQ$. 
        By \cite{Lefschetzclasses}*{Prop. 4}, all threefolds are of Lefschetz type. 
        As a very special case, we obtain stability conditions with full support on all smooth and projective Calabi--Yau threefolds.
        \vspace{-2mm}
        \item Smooth projective toric varieties have algebraic cohomology integrally generated in degree $2$ --- see \cite{danilov}*{Cor. 10.7.2} and the subsequent discussion or \cite{FultonToricvarieties}*{p. 106}. \vspace{-2mm}
        \item 
        By 
        \cite{Lefschetzclasses}*{Prop. 5} all smooth complete intersections in $\PP^n$ are of Lefschetz type.
        
        Smooth complete intersections of odd dimension have algebraic cohomology generated by the ample class and thus are of Lefschetz type. The same is true for very general complete intersection of even dimension, excluding quadrics, $(2,2)$-complete intersections and cubic surfaces --- see \cite{SGA7_II}*{Thm. 1.3}. \vspace{-2mm} 
        \item Abelian varieties are of Lefschetz type by \cite{Lefschetzclasses}*{Prop. 4}.
        Very general Abelian varieties have algebraic cohomology generated by an ample divisor --- see \cite{Mattuck}.
        \item All wonderful varieties $X$ satisfy $L^\bullet_X \otimes_{\ZZ}\QQ=\H^\bullet_{\rm{alg}}(X)_\QQ$  --- see \cites{SemismallBHMPW,dCP}. Furthermore, they are of Lefschetz type by the results of \cite{Lefschetzclasses}. Notable examples of varieties in this class are the Deligne--Mumford--Knudsen moduli spaces $\overline{M}_{0,n}$ of genus $0$ curves with $n$ marked points. \vspace{-2mm}
        \item Suppose that $X$ satisfies $L^\bullet_X \otimes_{\ZZ}\QQ=\H^\bullet_{\rm{alg}}(X)_\QQ$ and let $E$ be an algebraic vector bundle over $X$. The projective bundle $\bP(E) \to X$ also satisfies $L^\bullet_{\bP(E)} \otimes_{\ZZ}\QQ=\H^\bullet_{\rm{alg}}(\bP(E))_\QQ$ by the Leray--Hirsch theorem. Furthermore, if $X$ satisfies Standard conjecture A then so does $\bP(E)$ so that $\bP(E)$ is of Lefschetz type.  Similarly, the class of varieties for which Corollary~\ref{C:fullsupport} applies is closed under products. \vspace{-2mm}
    \end{enumerate}
\end{ex}

Though Theorem~\ref{T:maintheorem} and Corollary~\ref{C:fullsupport} give many new examples of varieties admitting stability cond\-itions with full support, many varieties are not of Lefschetz type \cite{Lefschetzclasses} or do not have algebraic cohomology generated by divisor classes. However, the following conjecture implies the existence of stability conditions with full support on all smooth projective varieties for which the restriction of the Poincaré pairing on $\H^\bullet_{\rm{alg}}(X)_\QQ$ is non-degenerate  --- see Proposition~\ref{P:grassmannian}. 
Note that the Standard conjecture A implies that 
the Poincaré pairing on algebraic cohomology classes is non-degenerate.

\begin{conj}
\label{C:grassmannian_result}
    Given a product of Grassmannians $Y:=\prod_i\Gr(n_i,k_i)$ and $N\in \bZ_{\ge 1}$, there exists a geometric stability condition with full support on $Y$ satisfying conditions (1) and (2) of Theorem \ref{T:maintheorem}.
    %the Bayer property and the restriction-$N$ property.
\end{conj}
\begin{prop}
\label{P:grassmannian}
    If Conjecture~\ref{C:grassmannian_result} holds, then all smooth projective varieties $X$, such that $\H^\bullet_{\rm{alg}}(X)_\QQ$ satisfies Poincaré duality, admit geometric stability conditions with full support.
\end{prop}

Before the proof of Proposition~\ref{P:grassmannian} let us recall some standard facts about algebraic classes, (strongly) ample vector bundles, and Grassmannians.

We say that a vector bundle $\cE$ of rank $r$ on a variety $X$ is \emph{strongly ample} if for all (closed) points $x\in X$, the sheaf $\cE\otimes \cI_x$ is globally generated.\footnote{Here, $\cI_x$ is the ideal sheaf of $x$.} By \cite{FultonAmpleVB}*{p. 172}, this condition implies that $\cE$ is ample in the sense of Hartshorne \cite{AmpleVBHartshorne} and that the induced morphism $f:X\to \Gr(\H^0(X,\cE),r)$,  mapping $x\in X$ to the quotient $\H^0(X,\cE)\twoheadrightarrow \cE_x$, is a closed immersion such that $f^*\cQ = \cE$, where $\cQ$ is the tautological quotient bundle on $\Gr(\H^0(X,\cE), r)$.

\begin{lem}
\label{L:veryample}
    Let $\cE$ be a vector bundle on a variety $X$ and $\cL$ an ample line bundle. There exists $N\ge 1$ such that $\cE\otimes \cL^{\otimes N}$ is strongly ample.
\end{lem}

\begin{proof}
    We consider the product $X\times X$ and denote the projections by $p_i$ for $i=1,2$. Given a closed point $x$ of $X$, we write $i_x:X\to X\times X$ for the inclusion of the fibre $p_1^{-1}(x)$. The line bundle $\cL\boxtimes \cL$ is ample on $X\times X$ and thus $\cF^N := \cI_\Delta \otimes p_2^*\cE \otimes (\cL^{\otimes N} \boxtimes \cL^{\otimes N})$ is generated by global sections for $N$ sufficiently large. Thus, restricting along $i_x^*$, we see that 
    \[
        i_x^*\cF^N = \cI_x \otimes \cE \otimes \cL_x^{\otimes N} \otimes \cL^{\otimes N} \cong \cI_x\otimes \cE \otimes \cL^{\otimes N}
    \]
    is generated by global sections for any $x\in X$.
\end{proof}

\begin{proof}[Proof of Proposition~\ref{P:grassmannian}]
    Consider a smooth projective variety $X$. We first show that $\H^\bullet_{\rm{alg}}(X)_{\bQ}$ is generated as a $\bQ$-algebra by the Chern classes of strongly ample vector bundles. The vector space $\H^\bullet_{\rm{alg}}(X)_{\bQ}$ is spanned by the Chern characters of finitely many algebraic vector bundles, so consider $\ch(\cE)$ for $\cE$ a vector bundle on $X$. Let $\cL$ denote a very ample line bundle on $X$. By Lemma~\ref{L:veryample}, there exists an $n\ge 1$ such that $\cE\otimes \cL^{n}$ is very ample. Furthermore,
    \[
        e^{-nc_1(\cL)}\cdot \ch(\cE\otimes \cL^n) = \ch(\cE) 
    \]
    and thus $\ch(\cE)$ is expressed as a $\bQ$-linear combination of products of Chern classes of strongly ample vector bundles and very ample line bundles. Thus, there exists a finite set $\cE_1,\ldots, \cE_p$ of strongly ample vector bundles on $X$ whose Chern characters generate $\H^\bullet(X,\bQ)$ as a $\bQ$-algebra. Consider the closed immersion $f:X\to\prod_{i=1}^p\Gr(\H^0(X,\cE_i),r_i) =: G$ defined by $\cE_1\oplus \cdots \oplus \cE_p$, where $r_i = \rk \cE_i$. Letting $\cQ_i\to \Gr(\H^0(X,\cE_i),r_i)$ denote the universal quotient bundle, it follows that under the closed immersion $f_i:X\to \Gr(\H^0(X,\cE_i), r_i)$ we have $f_i^*(\cQ_i) = \cE_i$ for each $1\le i \le p$. Thus, $f^*:\H^\bullet(G,\bQ) \to \H^\bullet(X,\bQ)$ maps surjectively onto $\H^\bullet_{\rm{alg}}(X)_{\bQ}$.

    Next, if Conjecture~\ref{C:grassmannian_result} holds, we can apply 
    Theorem~\ref{T:maintheorem} to $f$ to obtain geometric stability conditions on $X$ supported on $(\cH^\bullet_M(f),f_*\circ v_X)$. On the other hand, $f_*$ is injective on $\H_{\rm{alg}}^\bullet(X)_{\bQ}$ by Lemma~\ref{L:restrictioninjective} as we assumed Poincar\'e duality on algebraic classes, and thus also on $\cH^\bullet_M(X)$. The full support condition now follows.
\end{proof}


%\bibliographystyle{alpha}
%\bibliography{refs}

\begin{bibdiv}
\begin{biblist}

\bib{ABLStabilitysurf}{article}{
      author={Arcara, Daniele},
      author={Bertram, Aaron},
       title={Bridgeland-stable moduli spaces for {$K$}-trivial surfaces},
    language={eng},
        date={2013},
     journal={Journal of the European Mathematical Society},
      volume={015},
      number={1},
       pages={1\ndash 38},
         url={http://eudml.org/doc/277696},
}

\bib{AspinwallDouglasDbrane}{article}{
      author={Aspinwall, Paul},
      author={Douglas, Michael},
       title={D-brane stability and monodromy},
        date={200205},
     journal={Journal of High Energy Physics},
      volume={2002},
       pages={031\ndash 031},
}

\bib{Bayer_short}{article}{
      author={Bayer, Arend},
       title={A short proof of the deformation property of {B}ridgeland stability conditions},
        date={2019},
     journal={Mathematische Annalen},
      volume={375},
      number={3},
       pages={1597\ndash 1613},
         url={https://doi.org/10.1007/s00208-019-01900-w},
}

% \bib{Blanc_2015}{article}{
%       author={Blanc, Anthony},
%        title={Topological {K}-theory of complex noncommutative spaces},
%         date={2015},
%         ISSN={1570-5846},
%      journal={Compositio Mathematica},
%       volume={152},
%       number={3},
%        pages={489–555},
%          url={http://dx.doi.org/10.1112/S0010437X15007617},
% }

\bib{Br07}{article}{
      author={Bridgeland, Tom},
       title={Stability conditions on triangulated categories},
        date={2007},
        ISSN={0003-486X},
     journal={Ann. of Math. (2)},
      volume={166},
      number={2},
       pages={317\ndash 345},
         url={https://doi.org/10.4007/annals.2007.166.317},
      review={\MR{2373143}},
}

\bib{BridgelandK3}{article}{
      author={Bridgeland, Tom},
       title={Stability conditions on {K}3 surfaces},
    language={English},
        date={2008},
        ISSN={0012-7094},
     journal={Duke Math. J.},
      volume={141},
      number={2},
       pages={241\ndash 291},
}

\bib{SemismallBHMPW}{article}{
      author={Braden, Tom},
      author={Huh, June},
      author={Matherne, Jacob~P.},
      author={Proudfoot, Nicholas},
      author={Wang, Botong},
       title={A semi-small decomposition of the {Chow} ring of a matroid},
    language={English},
        date={2022},
        ISSN={0001-8708},
     journal={Adv. Math.},
      volume={409},
       pages={49},
        note={Id/No 108646},
}

\bib{BMunreasonable}{incollection}{
      author={Bayer, Arend},
      author={Macr{\`{\i}}, Emanuele},
       title={The unreasonable effectiveness of wall-crossing in algebraic geometry},
    language={English},
        date={2023},
   booktitle={International congress of mathematicians 2022, icm 2022, {H}elsinki, {F}inland, virtual, {J}uly 6--14, 2022. volume 3. sections 1--4},
   publisher={Berlin: European Mathematical Society (EMS)},
       pages={2172\ndash 2195},
}

\bib{ProjbiratBM}{article}{
      author={Bayer, Arend},
      author={Macrì, Emanuele},
       title={Projectivity and birational geometry of {B}ridgeland moduli spaces},
        date={2014},
        ISSN={08940347, 10886834},
     journal={Journal of the American Mathematical Society},
      volume={27},
      number={3},
       pages={707\ndash 752},
         url={http://www.jstor.org/stable/43302856},
}

\bib{stabilityonFanothreefolds}{article}{
      author={Bernardara, Marcello},
      author={Macrì, Emanuele},
      author={Schmidt, Benjamin},
      author={Zhao, Xiaolei},
       title={Bridgeland stability conditions on {F}ano threefolds},
        date={2017 September},
        ISSN={2491-6765},
     journal={Épijournal de Géométrie Algébrique},
      volume={Volume 1},
         url={http://dx.doi.org/10.46298/epiga.2017.volume1.2008},
}


\bib{YCheng}{article}{
    author={Cheng, Yiran},
    title={A remark on the full support property},
    language={English},
    date={2026-08},
    journal={arXiv e-prints},
    pages={arXiv:2608.14540},
}
    

\bib{dCP}{article}{
        author={Procesi, Claudio},
      author={De~Concini, Corrado},
       title={Wonderful models of subspace arrangements},
    language={English},
        date={1995},
        ISSN={1022-1824},
     journal={Sel. Math., New Ser.},
      volume={1},
      number={3},
       pages={459\ndash 494},
}

\bib{danilov}{article}{
      author={Danilov, Vladimir},
       title={The geometry of toric varieties},
        date={197804},
     journal={Russian Mathematical Surveys},
      volume={33},
       pages={97\ndash 154},
}

\bib{SGA7_II}{article}{
      author={Deligne, Pierre},
       title={Le théorème de {N}oether},
        date={1969},
     journal={SGA7 II, Exposé XIX},
}

\bib{DouglasDbranes}{incollection}{
      author={Douglas, Michael~R.},
       title={D-branes on {Calabi}-{Yau} manifolds},
    language={English},
        date={2001},
   booktitle={3rd european congress of mathematics (ECM), barcelona, spain, july 10--14, 2000. volume ii},
   publisher={Basel: Birkh{\"a}user},
       pages={449\ndash 466},
}

\bib{Mukaisprogram}{article}{
      author={Feyzbakhsh, Soheyla},
       title={Mukai's program (reconstructing a {{\(K3\)}} surface from a curve) via wall-crossing},
    language={English},
        date={2020},
        ISSN={0075-4102},
     journal={J. Reine Angew. Math.},
      volume={765},
       pages={101\ndash 137},
}

\bib{FultonAmpleVB}{article}{
      author={Fulton, William},
       title={Ample vector bundles, {Chern} classes, and numerical criteria},
    language={English},
        date={1976},
        ISSN={0020-9910},
     journal={Invent. Math.},
      volume={32},
       pages={171\ndash 178},
         url={https://eudml.org/doc/142370},
}

\bib{FultonToricvarieties}{book}{
      author={Fulton, William},
       title={Introduction to toric varieties. {The} 1989 {William} {H}. {Roever} lectures in geometry},
    language={English},
      series={Ann. Math. Stud.},
   publisher={Princeton, NJ: Princeton University Press},
        date={1993},
      volume={131},
        ISBN={0-691-00049-2},
}

\bib{FLZ}{article}{
      author={Fu, Lie},
      author={Li, Chunyi},
      author={Zhao, Xiaolei},
       title={Stability manifolds of varieties with finite {Albanese} morphisms},
    language={English},
        date={2022},
        ISSN={0002-9947},
     journal={Trans. Am. Math. Soc.},
      volume={375},
      number={8},
       pages={5669\ndash 5690},
}

\bib{GGI}{article}{
    author = {Galkin, Sergey},
    author = {Golyshev, Vasily},
    author = {Iritani, Hiroshi},
    title= {Gamma classes and quantum cohomology of Fano manifolds: Gamma conjectures},
    language={English},
    date={2016},
    ISSN={0012-7094},
    journal={Duke Mathematical Journal},
    volume={165},
    number={11},
    }


\bib{NMMP}{article}{
      author={{Halpern-Leistner}, Daniel},
       title={{The noncommutative minimal model program}},
        date={2023-01},
     journal={arXiv e-prints},
       pages={arXiv:2301.13168},
}

\bib{AmpleVBHartshorne}{article}{
      author={Hartshorne, Robin},
       title={Ample vector bundles},
    language={English},
        date={1966},
        ISSN={0073-8301},
     journal={Publ. Math., Inst. Hautes {\'E}tud. Sci.},
      volume={29},
       pages={63\ndash 94},
         url={https://eudml.org/doc/103864},
}

\bib{quasiconvergence}{article}{
      author={Halpern-Leistner, Daniel},
      author={Jiang, Jeffrey},
      author={Robotis, Antonios-Alexandros},
       title={Quasi-convergence of stability conditions},
        date={2026},
     journal={Sel. Math., New Ser., to appear},
}

\bib{augmented}{article}{
    author={Halpern-Leistner, Daniel},
    author={Robotis, Antonios-Alexandros},
      title={The space of augmented stability conditions}, 
      date={2025},
      url={https://arxiv.org/abs/2501.00710},
      note={arXiv:2501.00710}
}

\bib{Lefschetzclasses}{article}{
      author={Huh, June},
      author={Wang, Botong},
       title={Lefschetz classes on projective varieties},
    language={English},
        date={2017},
        ISSN={0002-9939},
     journal={Proc. Am. Math. Soc.},
      volume={145},
      number={11},
       pages={4629\ndash 4637},
}

\bib{karube}{article}{
      author={{Karube}, Tomohiro},
       title={{The noncommutative MMP for blowup surfaces}},
        date={2024-10},
     journal={arXiv e-prints},
       pages={arXiv:2410.18446},
}

\bib{Kleimanstandard}{incollection}{
      author={Kleiman, Steven~L.},
       title={The standard conjectures},
    language={English},
        date={1994},
   booktitle={Motives. {P}roceedings of the {S}ummer {R}esearch {C}onference on {M}otives, held at the university of {W}ashington, {S}eattle, WA, USA, july 20-august 2, 1991},
   publisher={Providence, RI: American Mathematical Society},
       pages={3\ndash 20},
}

\bib{KRZ}{article}{
      author={{Karube}, Tomohiro},
      author={{Robotis}, Antonios-Alexandros},
      author={{Zuliani}, Vanja},
       title={{Toward the noncommutative minimal model program for Fano varieties}},
        date={2026-01},
     journal={arXiv e-prints},
       pages={arXiv:2601.20739},
}

\bib{KS08}{article}{
      author={Kontsevich, Maxim},
      author={Soibelman, Yan},
       title={Stability structures, motivic {D}onaldson-{T}homas invariants and cluster transformations},
        date={2008-11},
        journal={arXiv e-prints},
       pages={arXiv:0811.2435},
         url={https://arxiv.org/abs/0811.2435},
}

\bib{LiQuintic}{article}{
      author={{Li}, Chunyi},
       title={On stability conditions for the quintic threefold},
        date={2019-05},
        ISSN={1432-1297},
     journal={Inventiones mathematicae},
      volume={218},
      number={1},
       pages={301–340},
         url={http://dx.doi.org/10.1007/s00222-019-00888-z},
}

\bib{Lirealreduction}{article}{
      author={{Li}, Chunyi},
       title={{A Real Reduction of the Manifold of Bridgeland Stability Conditions}},
        date={2025-06},
     journal={arXiv e-prints},
       pages={arXiv:2506.21995},
}

\bib{ChunyiLiRemark}{article}{
      author={{Li}, Chunyi},
       title={{A Remark on Stability Conditions on Smooth Projective Varieties}},
        date={2026-01},
     journal={arXiv e-prints},
       pages={arXiv:2601.22994},
}

\bib{2026stab_cond_modulispaces}{article}{
      author={{Li}, Chunyi},
      author={{Liu}, Zhiyu},
      author={{Liu}, Ziqi},
      author={{Macr{\`\i}}, Emanuele},
      author={{Perry}, Alexander},
      author={{Stellari}, Paolo},
      author={{Zhao}, Xiaolei},
       title={{Stability conditions and moduli spaces on projective families}},
        date={2026-07},
     journal={arXiv e-prints},
       pages={arXiv:2607.28411},
}

\bib{Lietal}{article}{
      author={Li, Chunyi},
      author={Macr{\`{\i}}, Emanuele},
      author={Perry, Alexander},
      author={Stellari, Paolo},
      author={Zhao, Xiaolei},
       title={Stability conditions on products of curves and {Hilbert} schemes of surfaces},
         how={Preprint, {arXiv}:2512.14207 [math.{AG}] (2025)},
        date={2025-12},
        journal={arXiv e-prints},
       pages={arXiv:2512.14207},
         url={https://arxiv.org/abs/2512.14207},
}

\bib{Macricurves}{article}{
      author={Macrì, Emanuele},
       title={Stability conditions on curves},
        date={200706},
     journal={Mathematical Research Letters},
      volume={14},
}

\bib{Mattuck}{article}{
      author={Mattuck, Arthur},
       title={Cycles on Abelian varieties},
        date={1958},
        ISSN={0002-9939,1088-6826},
     journal={Proc. Amer. Math. Soc.},
      volume={9},
       pages={88\ndash 98},
         url={https://doi.org/10.2307/2033404},
      review={\MR{98752}},
}

\bib{MMSinducing}{article}{
      author={Macrì, Emanuele},
      author={Mehrotra, Sukhendu},
      author={Stellari, Paolo},
      journal={J. Algebraic Geometry},
       title={Inducing stability conditions},
        date={2009-10},
        ISSN={1056-3911, 1534-7486},
      volume={18},
      number={4},
       pages={605\ndash 649},
}

\bib{Polishchuk}{article}{
      author={Polishchuk, Alexander},
       title={Constant families of $t$-structures on derived categories of coherent sheaves},
    language={eng},
        date={2007},
     journal={Moscow Mathematics Journal},
      volume={7},
      number={2},
       pages={109\ndash 134},
}

% \bib{wu2026gnoncomm}{article}{
%       author={{Wu}, Dongjian},
%       author={{Zhang}, Nantao},
%        title={{The $G$-Noncommutative Minimal Model Program}},
%         date={2026-02},
%      journal={arXiv e-prints},
%        pages={arXiv:2602.20335},
%}

\bib{zuliani}{article}{
      author={Zuliani, Vanja},
       title={Semiorthogonal decompositions of projective spaces from small quantum cohomology},
        date={2026},
     journal={Kyoto J. Math., to appear},
         url={https://arxiv.org/abs/2406.17616},
        note={arXiv:2406.17616},
}



\end{biblist}
\end{bibdiv}

\end{document}